\documentclass[11pt]{amsart}

\usepackage{amssymb,amsmath,amsthm, url}
\usepackage[all]{xy}

\xyoption{all}

\newcommand{\bbA}{\mathbb{A}}
\newcommand{\bbP}{\mathbb{P}}

\newtheorem{theorem}{Theorem}[section]

\newtheorem{proposition}[theorem]{Proposition}

\newtheorem{lemma}[theorem]{Lemma}

\newtheorem{cor}[theorem]{Corollary}

\theoremstyle{plain}
\numberwithin{equation}{section}

\theoremstyle{definition}

\newtheorem{remark}[theorem]{Remark}
\newtheorem{example}[theorem]{Example}

\newcommand{\C}{{\mathbb C}}

\DeclareMathOperator{\Spec}{Spec}

\DeclareMathOperator{\trdeg}{trdeg}

\newcommand{\isomto}{\overset{\sim}{\rightarrow}}

\newcommand{\A}{{\mathbb A}}

\newcommand{\bbZ}{{\mathbb Z}}

\newcommand{\rank}{\operatorname{rank}}
\newcommand{\Hilb}{\operatorname{Hilb}}

\begin{document}

\title{On a dynamical version of a theorem of Rosenlicht}

\author{J.~P.~Bell}
\address{
Jason Bell\\
Department of Pure Mathematics\\
University of Waterloo\\
Waterloo, ON N2L 3G1\\
CANADA 
}
\email{jpbell@uwaterloo.ca}

\author{D.~Ghioca}
\address{
Dragos Ghioca\\
Department of Mathematics\\
University of British Columbia\\
Vancouver, BC V6T 1Z2\\
Canada
}
\email{dghioca@math.ubc.ca}

\author{Z.~Reichstein}
\address{
Zinovy Reichstein\\
Department of Mathematics\\
University of British Columbia\\
Vancouver, BC V6T 1Z2\\
Canada
}
\email{reichst@math.ubc.ca}
\begin{abstract} Consider the action of an algebraic group $G$ on 
an irreducible algebraic variety $X$ all defined over a field $k$. 
M.~Rosenlicht showed that orbits in general position in $X$ can 
be separated by rational invariants.  We prove a dynamical analogue 
of this theorem, where $G$ is replaced by a semigroup of dominant
rational maps $X \dasharrow X$.  Our semigroup $G$ 
is not required to have the structure of an algebraic variety
and can be of arbitrary cardinality.
\end{abstract}

\maketitle

\section{Introduction}

The starting point for this paper is the following  
classical theorem of M.~Rosenlicht~\cite[Theorem 2]{rosenlicht}.

\begin{theorem} \label{thm.rosenlicht-algebraic} 
Consider the action of an algebraic group $G$ on 
an irreducible algebraic variety $X$ defined over a field $k$. 

\smallskip
(a) There exists a $G$-invariant dense open subset 
$X_0 \subset X$ and a $G$-equivariant morphism
$\phi \colon X_0 \to Z$ (where $G$ acts trivially on $Z$), 
with the following property.
For any field extension $K/k$ and any $K$-point $x \in X_0(K)$,
the orbit $G \cdot x$ equals the fiber $\phi^{-1}(\phi(x))$.

\smallskip
(b) Moreover, the field of invariants $k(X)^G$ is a purely inseparable
extension of $\phi^* k(Z)$, and one can choose $Z$ and $\phi$ so that 
$\phi^* k(Z) = k(X)^G$ (in characteristic zero, this is automatic).
\end{theorem}

In short, {\em for points $x, y$ in general position in $X$,
distinct $G$-orbits $G \cdot x$ and $G \cdot y$ can be separated 
by rational $G$-invariant functions}. Note, in particular, that
$G$-orbits in $X_0$ are closed in $X_0$. 

The rational map $\phi \colon X \dasharrow Z$, with
$k(Z) = k(X)^G$, is unique up to birational isomorphism. It is
called {\em the rational quotient} for the $G$-action on $X$.
For details and further references on this construction and its 
applications, see~\cite[Chapter 2]{popov-vinberg}.

The purpose of this note is to prove a dynamical version of 
this result, where the algebraic group $G$ is replaced by a
semigroup of dominant rational maps $X \dasharrow X$.  Note 
the semigroup $G$ is not required to have the structure of 
an algebraic variety,  and can be of arbitrary cardinality.
We will say that a closed (or open) subvariety $Y \subset X$
is $g$-invariant, for some $g \in G$, if $g(Y {\bf -} D_g) \subset Y$.
Here $D_g$ is the indeterminacy locus of $g \colon X \dasharrow X$.   
As usual, we will say that $Y$ is $G$-invariant if it is $g$-invariant
for every $g \in G$.
 
\begin{theorem} \label{thm.main}
Let $k$ be a field, $X$ be an irreducible quasi-projective
$k$-variety, and $G$ be a semigroup of dominant rational 
maps $X \dasharrow X$ defined over $k$.

Then there exists a dense open subvariety $X_0$,
a countable collection of closed $G$-invariant subvarieties 
$Y_1, Y_2, \dots \subsetneq X_0$ and a 
dominant morphism $\phi \colon X_0 \to Z$ 
with the following properties.

\smallskip
(a) $\phi \circ g = \phi$, as rational maps $X \dasharrow Z$.

\smallskip
(b) Let $K/k$ be a field extension and $x, y \in X_0(K)$ be $K$-points 
which do not lie in the indeterminacy locus of any $g \in G$,   
or on $Y_i$ for any $i \geqslant 1$. Then
$\phi(x) = \phi(y)$ if and only if 
$\overline{G \cdot x} = \overline{G \cdot y}$ in $X_K$.

\smallskip
(c) The field of invariants $k(X)^G$ is a purely inseparable
extension of $\phi^* k(Z)$. Moreover, one can choose $Z$ and $\phi$ 
so that $\phi^* k(Z) = k(X)^G$ (in characteristic zero, this 
is automatic).

\smallskip
Furthermore, if $G$ is a monoid (i.e., contains the identity 
morphism $X \to X$) then

\smallskip
(d) $X_0$ can be chosen to be $g$-invariant for every invertible 
element $g \in G$, and

\smallskip
(e) If $x \in X_0(K)$ is as in part (b), 
then the fiber $\phi^{-1}(\phi(x))$ of $x$ in $X_0$ equals
$\overline{(G \cdot x) \cap X_0}$.
\end{theorem}

In short, {\em for points $x$, $y$ in very general position in $X$,
distinct orbit closures $\overline{G \cdot x}$ and $\overline{G \cdot y}$
can be separated by rational $G$-invariant functions}. Here, as usual, 
``very general position" means ``in a countable intersection of dense 
open subsets".  Several remarks are in order.

\smallskip
(1) In the case where $G \isomto \mathbb{N}$ is generated by a single 
dominant rational map $\sigma \colon X \dasharrow X$
and $k = \mathbb C$,
Theorem~\ref{thm.main} was proved by E.~Amerik and 
F.~Campana~\cite[Theorem 4.1]{Amerik-Campana} by using
techniques of K\"ahler geometry.  Our proof of
Theorem~\ref{thm.main} is purely algebraic;
in particular, it is valid in prime characteristic.

\smallskip
(2) A conjecture of A.~Medvedev 
and Th.~Scanlon~\cite[Conjecture 7.14]{medvedev-scanlon}
asserts that in the case where $k$ is algebraically closed of characteristic $0$,
$G \isomto \mathbb N$ is generated by a single regular endomorphism $X \to X$ and
$k(X)^G = k$ (i.e., $Z$ is a point),
$X$ has a $k$-point with a dense $G$-orbit.
Over $\mathbb{C}$ the Medvedev-Scanlon conjecture follows from
the above-mentioned~\cite[Theorem 4.1]{Amerik-Campana}.
In the case where $k$ is any algebraically closed
uncountable field (of arbitrary characteristic), it was proved by the first author, D. Rogalski 
and S. Sierra~\cite[Theorem~1.2]{BRS}.  
We will reprove it in a stronger form (for an arbitrary semigroup $G$) 
as Corollary~\ref{cor.medvedev-scanlon-uncountable} below. 

Over a countable field, the Medvedev-Scanlon conjecture (which was, 
in turn, motivated by an earlier related conjecture 
of S.-W.~Zhang~\cite[Conjecture 4.1.6]{zhang-distributions}) remains
largely open. It has been established only in a small number of special 
instances (see, in particular,~\cite[Theorem 7.16]{medvedev-scanlon} 
and~\cite[Theorem 1.3]{recent-surface}), and no counterexamples 
are known.

\smallskip
(3) While we impose no restriction on the cardinality of the semigroup
$G$ in Theorem~\ref{thm.main}, the situation can be partially
reduced to the case, where $G$ is countable in the following sense.
In the setting of Theorem~\ref{thm.main} there always
exists a countable subsemigroup $H$ of $G$ such that $k(X)^G=k(X)^H$, 
and $\overline{G\cdot x}=\overline{H\cdot x}$ for $x \in X$ 
in very general position and away from the indeterminacy loci 
of every $g\in G$. For a precise statement, see 
Corollary~\ref{cor.countable}. On the other hand, there may not 
exist a finitely generated subgroup $H \subset G$ with these properties; 
see Example~\ref{ex.finitely-generated}. 

\smallskip
(4) If $X$ is not quasi-projective, Theorem~\ref{thm.main} can still 
be applied to any quasi-projective dense open subvariety 
$X' \subset X$. Note however, that replacing $X$ by a dense open 
subvariety $X'$ may make the condition on $x, y \in X_0(K)$ in part (b) 
more stringent by enlarging the indeterminacy loci of 
the elements $g \in G$, which we now view as dominant rational maps
$X' \dasharrow X'$.  

\smallskip
(5) The idea behind our construction of a rational map 
$\phi$ which separates $G$-orbit closures in very general position
in $X$ is as follows. Since $X$ is quasi-projective, we may 
assume that $X \subset \bbP^n$ for some $n \ge 1$.
We then set $\phi(x)$ to be to the class of the orbit closure 
$\overline{G \cdot x} \subset \bbP^n$ in the Hilbert scheme
$\Hilb(n)$  of subvarieties of $\bbP^n$.
The challenge is, of course, to show that this 
defines a rational map $\phi \colon X \dasharrow \Hilb(n)$.  
The ``quotient variety" $Z$ will then be defined as the closure
of the image of this map in $\Hilb(n)$.

Note that our proof of Theorem~\ref{thm.main} may be viewed as
an enhanced version of Rosenlicht's proof of 
Theorem~\ref{thm.rosenlicht-algebraic} in~\cite{rosenlicht}, where 
the ``separating map" $\phi$ sends $x \in X$ to the class of 
the orbit closure $\overline{G \cdot x}$ in the appropriate 
Chow variety. Leaving aside the difference between the Hilbert 
scheme and the Chow variety, we note that, as one may expect, the orbit 
closure $\overline{G \cdot x}$ varies 
more predictably with $x$ in the case of algebraic group actions than 
it does when $G$ is an arbitrary semigroup of dominant endomorphisms 
$X \dasharrow X$. For this reason, the construction of $\phi$
requires extra work in the dynamical setting 
(which is carried out in Section~\ref{sect.sections}), and
the resulting map only separates orbit closures in very general
position. In the last section we show how our arguments can be 
modified to yield a proof of 
Theorem~\ref{thm.rosenlicht-algebraic}.

\smallskip
(6) When it comes to separating orbit closures by rational invariants 
in the dynamical setting of Theorem~\ref{thm.main},
``very general position" is, indeed, the best one can do, even 
in the simplest case, where $G \isomto \mathbb N$ is a group generated by a single 
dominant morphism $\sigma \colon X \to X$.
Examples~\ref{ex1} and \ref{ex2} show that
if we replace the countable collection 
of $\{ Y_i, \, i \geq 1 \}$ of proper subvarieties of $X$ by 
a finite collection, Theorem~\ref{thm.main} will fail.
Note that in Example~\ref{ex1}, $\sigma$ is an automorphism.

\smallskip
(7) The problem of algorithmically computing a generating set 
for $k(X)^G$ (over $k$) for an action of a linear algebraic 
group  $G$ on a variety $X$, has received a lot of attention.  
For an overview and further references, see, 
e.g.,~\cite[Introduction]{kemper}.
We are not aware of any similar algorithms 
in the dynamical setting of Theorem~\ref{thm.main}, even when 
$G$ is generated by a single dominant morphism $X \to X$.

\section{A dense set of rational sections}
\label{sect.sections}

In this section we will consider the following situation:
\[ \xymatrix@R=4mm{ 
V_x \ar@{->}[ddd]_{\pi \;} & \subset & V   \ar@{->}[ddd]_{\pi \; } \\ 
 & & \\ 
 & & \\ 
x & \in & \; X \ar@{-->}@/_.7pc/[uuu]
\ar@{-->}@/_1.5pc/[uuu]_{\; \dots } 
\ar@{-->}@/_3pc/[uuu]_{\; \; s_{\lambda}, \; \lambda \in \Lambda} 
\; .} \] 
Here $V$ and $X$ are $k$-varieties, $X$ is irreducible, 
$\pi \colon V \to X$ is a $k$-morphism, and
$s_{\lambda} \colon X \dasharrow V$ is a collection of 
rational sections $X \dasharrow V$, indexed by a  
set $\Lambda$. The fiber $\pi^{-1}(x)$ of a point $x \in X$ 
will be denoted by $V_x$.
Note that we do not assume that $V$ is irreducible and do not
impose any restrictions on the cardinality of $\Lambda$.

If $K/k$ is a field extension, it will be convenient for us to denote 
by $X(K)'$ the collection of $K$-points of $X$ lying away from the 
indeterminacy locus of $s_{\lambda}$, for 
every $\lambda \in \Lambda$.  In other words, for $x \in X(K)'$,
$s_{\lambda}(x)$ is defined for every $\lambda \in \Lambda$. 
Note that if $\Lambda$ is large enough, 
$X(K)'$ may be empty for some fields $K/k$, even 
if $K$ is algebraically closed. On the other hand, 
the generic point $x_{\rm gen}$ of $X$ lies in $X(K_{\rm gen})'$,
where $K_{\rm gen} = k(X)$ is the function field of $X$.

\begin{proposition} \label{prop.sections}
Assume that the union of $s_{\lambda}(X)$ over all 
$\lambda \in \Lambda$ is dense in $V$.
Then there exists a countable collection $\{ Y_i, \, i \geqslant 1 \}$
of proper subvarieties of $X$ with the following property:
For any field extension $K/k$ and $x \in X(K)'$ 
away from $\bigcup_{i= 1}^{\infty} Y_i$, the set
\[ \{ s_{\lambda}(x) \, | \, \lambda \in \Lambda \} \]
is Zariski dense in the fiber $V_x := \pi^{-1}(x)$.
\end{proposition}

\begin{proof} We begin with two reductions. 

Since $k(V)$ is finitely generated over $k(X)$, there is 
a closed subvariety $V' \subset X \times \bbA^n$ and 
a birational isomorphism $j \colon V \dasharrow V'$ over $X$.
Note that replacing $V$ by $V'$ necessitates the removal of 
the sections $s_\lambda \colon X \dasharrow V$ whose image lies 
in the indeterminacy locus of the birational isomorphism 
$j$. This will not present a problem for us, since the union 
of the images of the remaining sections will still be dense in $V$. 

Secondly, by generic flatness (see, e.g.,~\cite[Theorem 14.4]{eisenbud}),
after replacing $X$ by a dense open subvariety, 
we may assume that $\pi$ is flat. Consequently,
the Hilbert function of the fiber $V_x$ is the same for every
$K$-point $x \in V_x$ and every field $K/k$; see, 
e.g.,~\cite[Exercise 6.11]{eisenbud}.  Let us denote this 
common Hilbert function by $h(d)$. That is,
if $I_x \subset K[t_1, \dots, t_n]$ is the ideal 
of $V_x$ in $\bbA^n$, and $I_x[d]$ is the $K$-vector space 
of homogeneous polynomials of degree 
$d$ in $I_x$, then $\dim_K(I_x[d]) = h(d)$.

With these reductions in place, we are ready to proceed with 
the proof.  Let $K/k$ be a field and $x \in X(K)'$.  Denote by $W_{x}$
the closure of $\{ s_{\lambda}(x) \, | \, \lambda \in \Lambda \}$ in $V$. 
Clearly 
\begin{equation} \label{e.inclusion}
W_x \subset V_x \subset \bbA^n_K \, .
\end{equation} 
Let $I(V_x)$ and $I(W_x) \subset K[t_1, \dots, t_n]$ 
be the ideals of $V_x$ and $W_x$ in $K[t_1, \dots, t_n]$, 
respectively. We will denote the $K$-vector space of polynomials
in $t_1, \dots, t_n$ of total degree $\leqslant d$
contained in $I(V_x)$ by $I(V_x)[d]$, and similarly for 
$I(W_x)[d]$. Note that by the above assumption on flatness, 
$\dim(I(V_x)[d]) = h(d)$
depends only on $d$ and not on the choice of $x$.

By~\eqref{e.inclusion}, 
 $I(V_x) \subset I(W_x)$. We claim that for every $d \geqslant 1$
there exists a proper closed subvariety $Y_d \subset X$ such that
$\dim(I(W_x)[d]) = h(d)$ for any field $K/k$ and any $x \in X(K)'$ 
away from $Y_d$.  If we can prove this claim, 
then for $x \in X(K)'$ away from $Y_1 \cup Y_2 \cup \dots$,
\[ I(W_x)[d] = h(d) = I(V_x)[d] \] 
for every $d \geq 1$ and thus $I(W_x) = I(V_x)$, so that
$W_x = V_x$, as desired.

To prove the claim, let us fix $d \geqslant 1$ and denote the distinct 
monomials in $t_1, \dots, t_n$ of degree $\leqslant d$ by $M_1, \dots, M_{l}$.
(Here $l = \binom{n + d}{d}$, but we shall not use this in the sequel.) 
Let $x_{\rm gen} \in X(K_{\rm gen})'$ be the generic point of
$X$, where $K_{\rm gen} = k(X)$ is the function field.
An element $\alpha(t_1, \dots, t_n) = \sum_{i = 1}^l \alpha_i M_i 
\in K_{\rm gen}[t_1, \dots, t_n]$ lies in $I(V_{x_{\rm gen}})[d]$ if and 
only if $\alpha$ vanishes on $V$. Since the images of the rational
sections $s_{\lambda} \colon X \dasharrow \bbA^n$  are dense in $V$, 
this is equivalent to $\alpha(s_{\lambda}) = 0$
in $k(X)$, for each $\lambda \in \Lambda$. The latter condition is,
in turn, equivalent to requiring $\alpha_1, \dots, \alpha_l \in K_{\rm gen}$
to satisfy the system
\begin{equation} \label{e.inf-system}  
\left\{ 
\begin{matrix}
\alpha_1 M_1(s_{\lambda}) + \alpha_2 M_2(s_{\lambda}) + \dots + 
\alpha_l M_{l}(s_{\lambda}) = 0 \\
\lambda \in \Lambda
\end{matrix} \right. 
\end{equation}
of linear equations over $K_{\rm gen} = k(X)$. 
Thus $h(d) := \dim_{K_{\rm gen}}(I_{x_{\rm gen}}[d])$, is the dimension
of the solution space to this system. If we denote the 
matrix of the system~\eqref{e.inf-system} by
\begin{equation} \label{e.inf-matrix}
A_{\rm gen} := \begin{pmatrix}
M_1(s_{\lambda}) & M_2(s_{\lambda}) & \dots & 
M_{l}(s_{\lambda}) \\ \lambda \in \Lambda
\end{pmatrix} \, ,   
\end{equation}
then $h(d) = l - r$, where $r := \rank(A_{\rm gen})$.
Here the columns of $A_{\rm gen}$ are indexed by $1, \dots, l$ and
the rows by $\lambda \in \Lambda$. (In particular, if $\Lambda$ 
happens to be uncountable, $A_{\rm gen}$ will have uncountably many rows.)

We now define $Y_{d} \subset X$ by the condition that
$\rank(A_{\rm gen}) < r$. 
That is, $Y_d$ is the intersection of the closed subvarieties of $X$
cut out by the system of equations 
\[  \left\{ 
\begin{matrix}
\det(B) = 0 \\
\text{$B$ ranges over the $r \times r$-submatrices of $A_{\rm gen}$.}  
\end{matrix} \right. \]
By our assumption, at least one of the $r \times r$-submatrices of
$A_{\rm gen}$ is non-singular.  Thus $Y_d$ is a proper closed subvariety
of $X$ defined over $k$.

Now suppose $K/k$ is a field extension and $x \in X(K)'$, as before.
A polynomial
$\beta(t_1, \dots, t_n) = \sum_{i = 1}^{l} \beta_i M_i 
 \in K[t_1, \dots, t_n]$ lies in $I(W_x)[d]$ 
if and only if $\beta_1, \dots, \beta_{l} \in K$ satisfy 
the linear system
\begin{equation} \label{e.inf-system2}  
\left\{ 
\begin{matrix}
\beta_1 M_1(s_{\lambda}(x)) + \beta_2 M_2(s_{\lambda}(x)) + \dots + 
\beta_l M_{l}(s_{\lambda}(x)) = 0 \\
\lambda \in \Lambda
\end{matrix} \right. 
\end{equation}
with coefficients in $K$. The matrix $A_x$ of this system is 
obtained from $A_{\rm gen}$ by evaluating each entry at $x$. (Note 
that it makes sense to evaluate an entry of $A_{\rm gen}$
at $x \in X(K)'$. Indeed, $x$ lies in the domain of each $s_{\lambda}$
and every entry in $A_{\rm gen}$ is a monomial in the coordinates of 
$s_{\lambda}(x) \in V_x \subset \bbA^n_K$.) 
Thus $\dim(I(W_x)[d]) = l - \rank(A_x)$.
If we choose $x$ away 
from $Y_d$, then $\rank(A_x) = \rank(A_{\rm gen}) = r$ and 
\[ \dim(I(W_x)[d]) = l - \rank(A_x) = l - r = h(d) = \dim(I(V_x)[d]) \, , \]
as desired.
\end{proof}

\section{The Hilbert scheme}
\label{sect.hilbert-scheme}

The Hilbert scheme $\Hilb(n)$, constructed 
by A.~Grothendieck~\cite{grothendieck}, classifies closed
subvarieties of $\bbP^n$ in the following sense. 
A family of subvarieties of $\bbP^n$ parametrized by
a scheme $X$ is, by definition, a closed subvariety 
\[ V \subset X \times \bbP^n \]
such that the projection $V \to X$ to the first factor is flat.
Families of subvarieties of $\bbP^n$ parametrized by $X$  
are in a natural (functorial in $X$) bijective correspondence
with morphisms $X \to \Hilb(n)$.  Note that $\Hilb(n)$ is
not a Noetherian scheme, it is only locally Noetherian. 
This will not present a problem for us though, because
$\Hilb(n)$ is a disjoint union of (infinitely many) 
schemes of the form $\Hilb(n, p(d))$, 
where $p(d)$ is a Hilbert polynomial, and 
each $\Hilb(n, p(d))$ is a projective variety defined over $\bbZ$.
If $X$ is an irreducible variety, and $V \to X$ is 
a family of subvarieties of $\bbP^n$, as above, then 
the image of the morphism $X \to \Hilb(X)$ associated to this family
lies in $\Hilb(n, p(d))$, where $p(d)$ is the Hilbert polynomial of any
fiber of $V \to X$. (Since the morphism $V \to X$ is assumed to be flat,
every fiber will have the same Hilbert polynomial.) 

We are now ready to proceed with the proof of Theorem~\ref{thm.main}.
In this section we will construct a dense open subset $X_0 \subset X$, 
a dominant morphism $\phi \colon X_0 \to Z$, and a countable collection
of proper $k$-subvarieties $Y_d \subset X_0$. We will check that
each $Y_d$ is $G$-invariant and defer the rest of the proof of 
Theorem~\ref{thm.main} to the next two sections.

By our assumption $X$ is a quasi-projective variety. In other words,
$X$ is a locally closed subvariety of some projective space $\bbP_k^n$. 
Let $V \subset X \times \bbP^n$ be the Zariski closure of the union of
the graphs of $g \colon X \dasharrow X \subset \bbP^n$, as 
$g$ ranges over $G$. Let $\pi \colon V \to X$ be the projection 
\begin{equation} \label{e.graph}
\xymatrix{  
V := \overline{\{ (x, g(x)) \, | \, x \in X \, , \; 
g \in G \}}  \ar@{->}[d]_{\pi} & \subset X \times \bbP^n  \\
X   } 
\end{equation}
to the first factor. Let $X_0 \subset X$ be the flat locus of $\pi$, i.e.  
the largest dense open subset of $X$ over which $\pi$ is flat.
(Recall that $X_0$ is dense in $X$ by generic flatness.) 
Let $V_0 := \pi^{-1}(X_0)$. 
We now view $V_0 \subset X_0 \times \bbP^n$ as a family of
subvarieties of $\bbP^n$ parametrized by $X_0$.
By the universal property of the Hilbert scheme $\Hilb(n)$, 
this family induces a morphism $\phi \colon X_0 \to \Hilb(n)$.
Denote the closure of the image of this morphism by $Z$.
If $K/k$ is a field extension and $x, y \in X_0(K)$ then
by our construction
\begin{equation} \label{e.fiber2}
\text{$\phi(x) = \phi(y)$ if and only if $V_x = V_y$.}
\end{equation}
Here we identify $\{ x \} \times \bbP^n_{K}$ and 
$\{ y \} \times \bbP^n_{K}$ with  $\bbP^n_{K}$. 

Each $g \in G$ gives rise to the (regular) section $s_g \colon X_0
\to V_0$ given by $x \mapsto (x, g(x))$, as $g$ ranges over $g$.
By the definition of $V$, the union of the images of these sections 
is dense in $V_0$.  Thus by Proposition~\ref{prop.sections}
there exists a countable collection
of proper $k$-subvarieties $Y_i \subset X_0$,
such that for any field extension $K/k$ and any
$x \in X_0(K)$ away from the union of these subvarieties,
\begin{equation} \label{e.fiber}
\text{$V_x = \overline{G \cdot x}$ in $\bbP^n_K$.}
\end{equation}

\begin{lemma} \label{lem.Y_d} 
Each subvariety $Y_d \subset X_0$ is $G$-invariant.
\end{lemma}

\begin{proof} We revisit the construction of $Y_d$ in the proof of
Proposition~\ref{prop.sections}. For the first reduction 
at the beginning of that proof, note that
$V$ is, by definition, a subvariety of $X \times \bbP^n$.
Thus we may choose an affine subspace $\bbA^n \subset \bbP^n$, 
so that $X$ intersects $\bbA^n$ non-trivially and replace $V$ by 
$V \cap \left( X\times \A^n \right)$.
The second reduction is carried out by passing from $X$ 
to $X_0$.

With these two reductions in place, $Y_d \subsetneq X_0$
was constructed as follows. Let $t_1, \dots, t_n$ 
be the affine coordinates on $\bbA^n$ and
$M_1, \dots, M_l$ be the distinct monomials of degree $d$ 
in $t_1, \dots, t_n$. Let 
\[ A_{\rm gen} := \begin{pmatrix}
M_1(g) & M_2(g) & \dots & 
M_{l}(g) \\ g \in G
\end{pmatrix} \, , \]
where we view $g$ as a rational function $X \dasharrow X \cap \bbA^n \subset 
\bbA^n$. Let $r$ be the rank of $A_{\rm gen}$. 
The subvariety $Y_{d} \subset X$ is then defined by the condition that
$\rank(A_{\rm gen}) < r$.  That is, $Y_d$ is the intersection 
of the closed subvarieties of $X_0$
cut out by the system of equations 
\[  \left\{ 
\begin{matrix}
\det(B_{i_1, \dots, i_r, g_1, \dots, g_r}) = 0 \\
\text{$1 \leqslant i_1 < \dots < i_r \leqslant l$ and $g_1, \dots, g_r$
are distinct elements of $G$.} 
\end{matrix} \right. \]
Here $B_{i_1, \dots, i_r, g_1, \dots, g_r}$ is the
$r \times r$-submatrix of $A_{\rm gen}$ consisting of
columns $i_1, \ldots, i_r$ and rows indexed by $g_1, \dots, g_r$.  

To prove the lemma we need to check that if a point 
$x \in X_0(K)$ lies on $Y_d$ then so does $g(x)$ 
whenever $g(x)$ is defined. Here $K/k$ is a field extension.
Indeed, 
\[ B_{i_1, \dots, i_r, g_1, \dots, g_r}(gx) = 
B_{i_1, \dots, i_r, g_1 g, \dots, g_r g}(x) \, . \]
If $x \in Y_d$ then the matrix on the right hand side is singular
for any choice of $1 \leqslant i_1, \dots, i_r \leqslant l$ and 
$g_1, \dots, g_r \in G$, as desired.
\end{proof}

\section{Conclusion of the proof of Theorem~\ref{thm.main}(a), (b), (d) and (e)}
\label{sect.main}

(b) By~\eqref{e.fiber2}, $\phi(x) = \phi(y)$ if and only if $V_x = V_y$ in 
$\bbP^n_K$.  By~\eqref{e.fiber}, 
$V_x = \overline{G \cdot x}$, $V_y = \overline{G \cdot y}$, where the closure
is taken in $\bbP^n_K$. This shows that $\phi(x) = \phi(y)$ if and only if 
$G \cdot x$ and $G \cdot y \subset X_K$ have 
the same closure in $\bbP^n_K$. On the other hand,
$G \cdot x$ and $G \cdot y$ have 
the same closure in $\bbP^n_K$ if and only if they have the same
closure in $X_K$.

\smallskip
(a) It suffices to show that the rational maps $\phi \circ g$ and $\phi
\colon X \dasharrow \Hilb(n)$ agree on the generic point 
$x_{\rm gen}$ of $X$ for every $g \in G$. Choose $g \in G$ and fix 
it for the rest of the proof. Then
$x := x_{\rm gen}$ and $y := g(x_{\rm gen})$
are $K_{\rm gen}$-points of $X$, where 
$K_{\rm gen} :=k(X)$. Since $G$ is dominant, neither $x$ nor $y$ lie on
any proper subvariety of $X$ defined over $k$. In particular, 
they do not lie in the indeterminacy locus of any $h \in G$ or on 
$Y_i$ for any $i \geqslant 1$. By \eqref{e.fiber2}, proving that
$\phi(x) = \phi(y)$ is equivalent to showing that
\begin{equation} \label{e.V_x}
\text{$V_x = V_y$ in $\bbP^n_{K_{\rm gen}}$,}
\end{equation} 
where $V_x, V_y \in \bbP^n_{K_{\rm gen}}$ are the fibers
of $x$ and $y$, under $\phi \colon V \to X$ in $\bbP^n_{K_{\rm gen}}$. 

Since $x, y \in X_0(K_{\rm gen})$ do not lie 
on $Y_i$ for any $i \geqslant 1$, part (b) tells us that
$V_x = \overline{G \cdot x}$ and $V_y = \overline{G \cdot y}$.
If $g$ is invertible in $G$, this immediately implies~\eqref{e.V_x}, 
since $G \cdot x = G \cdot y$.  In general (if $g$ is not necessarily 
invertible), we use the following alternative
argument to prove~\eqref{e.V_x}.  Since $y = g(x)$, we have 
$G \cdot y \subset G \cdot x$ and thus
\begin{equation} \label{e.orbits}
V_y \subset V_x \, . 
\end{equation}
Moreover, since $\pi \colon V \to X$ is flat over $X_0$, $V_x$ 
and $V_y$ have the same Hilbert function. Now~\eqref{e.V_x} follows
from~\eqref{e.orbits}.

\smallskip
(d) If $g$ is an invertible element of $G$ then one easily checks that
the variety $V \subset X \times \bbP^n$ defined in~\eqref{e.graph},
is $g$-invariant, where $g$ acts on $X \times \bbP^n$ via the first factor. 
Consequently, the flat locus $X_0 \subset X$ of the projection 
$\pi \colon V \to X$ is $g$-invariant. 

\smallskip
(e) By part (a), $\phi((G \cdot x) \cap X_0) = \phi(x)$ and thus
\begin{equation} \label{e.main(e)}
\overline{(G \cdot x) \cap X_0} \subset 
\phi^{-1}(\phi(x)),
\end{equation}
where the closure is taken in $X_0$.
On the other hand, if $y \in X_0$ and $\phi(y) = \phi(x)$ then by 
\eqref{e.fiber2}, $V_y = V_x$. Since $G$ is a monoid, $y \in V_y$ and
thus $y \in V_x$.
In other words, $\phi^{-1}(\phi(x)) \cap X_0 \subset V_x \cap X_0$.
Combining this with~\eqref{e.main(e)}, we obtain
\begin{equation} \label{e.main(e)'}
\overline{(G \cdot x) \cap X_0} \subset 
\phi^{-1}(\phi(x)) \subset V_x \cap X_0 \, . 
\end{equation}
On the other hand, by~\eqref{e.fiber}, $G \cdot x$ is dense in $V_x$.
Thus $(G \cdot x) \cap X_0$ is dense in $V_x \cap X_0$, i.e.,
\[ \overline{(G \cdot x) \cap X_0} = V_x \cap X_0 \, . \] 
We conclude that
both containments in~\eqref{e.main(e)'} are equalities. In particular,
$\overline{(G \cdot x) \cap X_0} = \phi^{-1}(\phi(x))$, as desired.

\section{Proof of Theorem~\ref{thm.main}(c)}
\label{sect.main(c)}

By part (a), $\phi^*(k(Z)) \subset k(X)^G$.  Let $Y$ be 
a $k$-variety whose function field $k(Y)$ is $k(X)^G$.
The inclusions 
 $k(Z) \stackrel{i^*}{\hookrightarrow} k(Y) = 
k(X)^G \stackrel{\psi^*}{\hookrightarrow} k(X) = k(X_0)$ 
 induce dominant rational maps
 \[ 
\xymatrix{
X_0 \ar@{->}[d]^{\; \; \psi}  \ar@{->}@/_1.5pc/[dd]_{\; \; \phi} \\  
Y \ar@{->}[d]^{\; \; i}    \\
Z. }
\]
After replacing $X_0$, $Z$, and $Y$ by suitable dense open
subvarieties, we may assume that all three maps in the above
diagram are regular. 

Let $K/k$ be a field and $x, y \in X_0(K)$ be as in 
Theorem~\ref{thm.main}(b). We claim that if $\phi(x) = \phi(y)$
then $\psi(x) = \psi(y)$. Indeed, by Theorem~\ref{thm.main}(b), 
$\phi(x) = \phi(y)$ implies that $\overline{G \cdot x} = \overline{G \cdot y}$
in $X_K$. Then 
\begin{equation} \label{e.thm.main(c)}
\text{$\overline{(G \cdot x) \cap X_0} = \overline{(G \cdot y) \cap X_0}$ 
in $X_0$.}
\end{equation}
By our construction, $\psi$ is $G$-equivariant, where $G$ acts trivially on
$Y$. Thus $\psi$ sends all of 
$\overline{(G \cdot x) \cap X_0}$ to the point $\psi(x)$ and all of 
$\overline{(G \cdot y) \cap X_0}$ to the point $\psi(y)$ in $Y$.
By~\eqref{e.thm.main(c)}, $\psi(x) = \psi(y)$, as claimed.

In particular, if $K/k$ is  field extension and $x, y \colon \Spec(K) \to X$
are dominant points, then $x$ and $y$ satisfy the conditions of
Theorem~\ref{thm.main}(b) and thus 
\[ \text{$\phi(x) = \phi(y)$ if and only if 
$\psi(x) = \psi(y)$.} \]
By Lemma~\ref{lem.dominant} below, $i$ is purely inseparable.
This proves the first assertion
of Theorem~\ref{thm.main}(d).

To prove the second assertion of part (d), we 
simply replace $Z$ by $Y$ and $\phi$ 
by $\psi$.  Since $i$ is inseparable, properties of $\phi$ 
asserted by Theorem~\ref{thm.main} are shared by $\psi$.

\begin{lemma} \label{lem.dominant}
Let $\phi \colon X \stackrel{\psi}{\to} Y \stackrel{i}{\to} Z$ be 
dominant maps of irreducible $k$-varieties.  
Suppose that for any pair of
dominant points $x, x' \colon \Spec(K) \to X$, where $K/k$ is 
a field extension, 
\begin{equation} \label{e.psi}
\text{$\phi(x) = \phi(x')$ if and only if $\psi(x) = \psi(x')$.}
\end{equation}
Then the field extension $k(Y)/i^*(k(Z))$ is purely inseparable. 
\end{lemma}

\begin{proof}
Let $F$ be the algebraic closure of
$k(X)$ and $x \colon \Spec(F) \to X$ be the dominant
$F$-point of $X$ obtained by composing the natural 
projection $\Spec(F) \to \Spec(k(X))$ with the
generic point $\Spec(k(X)) \to X$.  Set $z := \phi(x)$.
The fiber $\phi^{-1}(z)$ is an $F$-subvariety of $X_F$. 
Denote its irreducible components by $X_1, \dots, X_n$.  

By the fiber dimension theorem, the generic point
$x_i \colon \Spec(F(X_i)) \to X_i \hookrightarrow X$ 
is dominant for every $i = 1, \dots, n$.  
If $K/F$ is a compositum of $F(X_1), \ldots, F(X_n)$ over $F$
and $(x_i)_K$ is the composition of the projection
$\Spec(K) \to \Spec(F(X_i))$ with $x_i$, then 
\[ \phi((x_1)_K) = \dots = \phi((x_n)_K) = z_K = \phi(x_K) \, . \] 
Our assumption~\eqref{e.psi} now tells us that
$\psi((x_1)_K) = \dots = \psi((x_n)_K) = \psi(x_K)$.
Since $x$ is, by definition, an $F$-point of $X$, we see that 
$\psi((x_1)_K) = \dots = \psi((x_n)_K) = \psi(x)$ descends to a
dominant $F$-point $y \colon \Spec(F) \to Y$, where $i(y) = z$. 
In other words $\psi$ maps each $X_i$ to the single point
$y \in Y(F)$, as depicted in the following
diagram:
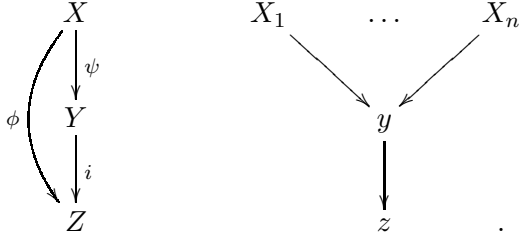
\[ \xymatrix{  
X \ar@{->}[d]^{\psi}  
\ar@{->}@/_1.5pc/[dd]_{\; \; \phi}   
 &  & X_1 \ar@{->}[dr] & \hdots & X_n \ar@{->}[dl] \\ 
Y \ar@{->}[d]^{i} & &  & y \ar@{->}[d] &  \\ 
Z &  &  & z & .}
\]
Thus $\psi(\phi^{-1}(z)) = y$. Equivalently,
$\psi(\psi^{-1}(i^{-1}(z))) = y$ or $i^{-1}(z) = y$. 
Applying the fiber dimension theorem one more time, we see that
$\dim(Y) = \dim(Z)$.  Since $z \colon \Spec(F) \to Z$ is dominant, and
$F$ is algebraically closed, the number of preimages of
$z$ in $Y$ equals the separability degree of $k(Y)$ over $i^*(k(Z))$.   
In our case the preimage of $z$ a single point $y \in Y(F)$; hence,
$k(Y)$ is purely inseparable over $i^* k(Z)$.
This completes the proof of Lemma~\ref{lem.dominant} and thus of
Theorem~\ref{thm.main}.
\end{proof}

  
\begin{remark} \label{rem.inseparable} 
We do not know whether or not $\phi^* k(Z)$ always coincides with
$k(X)^G$, where $Z$ is the closure of the image of
the rational map $\phi \colon X \dasharrow \Hilb(n)$ we constructed.
As we have just seen, this is always the case in characteristic zero, 
so the question is only of interest in prime characteristic. 
In the setting, where an algebraic group 
is acting rationally on $X$ (which is a bit more general than
that of Theorem~\ref{thm.rosenlicht-algebraic}), an analogous 
question concerning the map $\phi \colon X \dasharrow C$ into an appropriate 
Chow variety constructed by Rosenlicht, was left open 
in~\cite{rosenlicht}.  A.~Seidenberg~\cite{seidenberg} subsequently 
showed that, indeed, $k(C) = k(X)^G$ under the assumption that 
the $G$-action on $X$ 
is regular (as in Theorem~\ref{thm.rosenlicht-algebraic}).
\end{remark}

\section{Some corollaries and examples}

The following corollary of Theorem~\ref{thm.main} is a generalization of the
Medvedev-Scanlon conjecture~\cite[Conjecture 7.14]{medvedev-scanlon}
in the case, where the base field $k$ is uncountable.

\begin{cor} \label{cor.medvedev-scanlon-uncountable} 
Let $k$ be an uncountable algebraically closed field,
$X$ be an irreducible quasi-projective $k$-variety, 
and $G$ be a semigroup of dominant regular maps $X \to X$.
If $k(X)^G = k$ then $G \cdot x$ is dense in $X$ for some
$x \in X(k)$.  
\end{cor}

\begin{proof} Let $X_0 \subset X$ be a dense open subset, and
$\phi \colon X_0 \to Z$ be a morphism such that $\phi^* k(Z) = k(X)^G$,
as in Theorem~\ref{thm.main}.  Since $k(X)^G = k$, the variety $Z$ 
is a single point. 

Let $S \subset X_0(K)$ be the set of $k$-points
of $X_0$ away from the exceptional sets $Y_i$ for every $i \geqslant 1$. 
Since $k$ is algebraically closed and uncountable, 
$S$ is Zariski dense in $X_0$. By Theorem~\ref{thm.main}(b),
\begin{equation}
\label{e.common-orbit-closure}
\text{$Y = \overline{G \cdot x}$ is independent of the choice of $x \in S$.}
\end{equation}
Since $g \colon X \to X$ is dominant for every $g \in G$, we see that
the union
\[ \bigcup_{x \in S} \overline{G \cdot x} \]
is dense in $X$. On the other hand, by 
\eqref{e.common-orbit-closure}, this union equals $Y$, which is closed
in $X$. We conclude that $Y = X$, i.e.,
$\overline{G \cdot x} = X$ for every $x \in S$.
\end{proof}

Note that in the statement of Corollary~\ref{cor.medvedev-scanlon-uncountable}
we do not make any assumptions of the semigroup $G$; 
in particular, it may be uncountable. However, the following
corollary shows that $G$-orbit closures in very general position
are controlled by a countable subsemigroup $H \subset G$.

\begin{cor} \label{cor.countable}
Under the assumptions of Theorem~\ref{thm.main}, with $k(Z) = k(X)^G$,
there exists a countable subsemigroup $H$ of $G$ such that

\smallskip
(a) $k(X)^G = k(X)^H$, and

\smallskip
(b) Moreover, there exists a 
a countable collection $\{ W_i \, | \, i \geqslant 1 \}$ of $H$-invariant 
subvarieties of $X_0$ defined over $k$ with the following property. For 
any field extension $F/k$ and any $x \in X_0(F)$ away 
from the indeterminacy loci of all $g \in G$ and 
away from $\cup_{i \geqslant 1} \, W_i$, we have 
$\overline{G \cdot x} = \overline{H \cdot x}$.   
\end{cor}

\begin{proof}
Our proof will rely on the following elementary 

\smallskip
Claim: Let $W$ be an algebraic variety (not necessarily irreducible),
and $S$ be a dense collection of points in $W$. Then there exists a
countable subcollection $S' \subset S$ which is dense in $W$.

\smallskip
After replacing $W$ by a dense open subset,
and removing the points of $S$ that do not lie in this dense open subset,
we may assume without loss of generality that $W \subset \bbA^n$ is affine. 
Let $I(W)$ be the ideal of $W$ in $k[x_1, \dots, x_n]$ and 
$I(W)[d]$ be the finite-dimensional vector space of polynomials 
of degree $\leq d$ vanishing on $W$. It is easy to see that for 
each $d$ there is a finite subset $S_d \subset S$ such that 
$I(S_d)[d] = I(W)[d]$.  Taking $S' = \cup_{d=1}^{\infty} S_d$, we see
that $I(S')[d] = I(W)[d]$ for every $d \geqslant 1$. Thus
$I(S') = I(W)$ and $S'$ is dense in $W$. This proves the claim.
  
We now proceed with the proof of Corollary~\ref{cor.countable}.
We will assume that $X \subset \bbP^n$ is a locally closed 
subvariety of $\bbP^n$ for some $n \geq 1$. 
Let $V \subset X \times \bbP^n$ 
be the closure of the union of the images of 
\[ s_g \colon X \dasharrow X \times \bbP^n \, , \]
as in~\eqref{e.graph}.
Applying the claim to the generic points of the closures of the
images of $s_g$, as $g \in G$, we see that there exists a 
countable collection of elements $\{ h_i, \, | \, i \geqslant 1 \}$
such that the images of $s_{h_i}$ are dense in $V$. Let
$H$ be the countable subsemigroup of $G$ generated by 
these $h_i$. Denote  the flat locus of
the projection $\pi \colon V \to X$ by $X_0 \subset X$ and 
the morphism associated to $\pi$, viewed as a family of subvarieties 
over $X_0$, by $\phi \colon X_0 \to \Hilb(n)$.  Arguing as 
in Section~\ref{sect.hilbert-scheme},
we see that there exists a countable collection 
$\{ W_i \, | \, i \geqslant 1 \}$ of $H$-invariant
closed subvarieties of $X_0$ such that for any field extension
$K/k$ and any point $x \in X_0(K)$ away from each $W_i$ and 
from the indeterminacy locus of every $g \in G$, 
$V_x = \overline{H \cdot x}$ in $X$; cf.~\eqref{e.fiber}. Since 
$H \cdot x \subset G \cdot x \subset V_x$, we 
have $\overline{H \cdot x} = \overline{G \cdot x} = V_x$.
This proves part (b).

To prove part (a), note that by our construction, 
$\phi^* \, k(Z) \subset k(X)^G \subset k(X)^H$, and by
Theorem~\ref{thm.main}(c), $k(X)^H$ is purely inseparable over 
$\phi^* \, k(Z)$. Thus $k(X)^H$ is purely inseparable over $k(X)^G$.

It remains to show that, in fact, $k(X)^H = k(X)^G$.
Choose $a \in k(X)^H$. Then $a$ satisfies some purely inseparable polynomial
$p(t) \in k(X)^G[t]$. For any $g \in G$, $g(a)$ also satisfies 
$p(t)$. Since $a$ is the only root of $p(t)$ in
$k(X)$, we conclude that $g(a) = a$. In other words,
$a \in k(X)^G$, i.e., $k(X)^H = k(X)^G$, as desired.
\end{proof}

\begin{example} \label{ex.finitely-generated}
In general one cannot expect to find a finitely generated
subsemigroup $H \subset G$ such that $k(X)^G = k(X)^H$. 
For example, let $X$ be a complex abelian variety of 
dimension $\geqslant 1$, and $G$ 
be the group of translations on $X$ by torsion points of $X(\C)$. 
Then $G$ is countable and $\C(X)^G=\C$. On the other hand, 
any finitely generated subgroup $H$ of $G$ is finite 
and $\C(X)^H$ is a finite subextension of $\C(X)$.
\qed
\end{example}

The following examples show that the countable collection 
$\{ Y_i, \, i \geq 1 \}$ of proper ``exceptional"
subvarieties of $X$ cannot be replaced by
a finite collection in the statement of Theorem~\ref{thm.main}.

\begin{example} \label{ex1}
Let  $E$ be an elliptic curve over $k$, $X = E \times E$,
$\sigma$ be an automorphism of $X$ given by $\sigma(x, y) := (x, x + y)$,
and $G \isomto \mathbb N$ (or $G \isomto \mathbb Z$)
is generated by $\sigma$ as a semigroup (or as a group). 
 In this case $Z$ and $\phi$ are unique (up to birational 
isomorphism), and is easy to see that $\phi$ is just projection 
to the first factor,
$\phi \colon X \to Z := E$.  The fiber $X_z =  \{ x \} \times E$ is the closure 
of a single orbit if and only if $x$ is of infinite order in $E$. 
Thus the ``exceptional set" $Y_1 \cup Y_2 \cup \dots$ has to contain
countably many ``vertical" curves $\{ x \} \times E$, as $x$ ranges 
over the torsion points of $E$.
\end{example}

\begin{example} \label{ex2} Let $X = \bbP^n$ and $G \isomto \mathbb N$ be the cyclic 
semigroup generated by a single dominant morphism $\sigma \colon X \to X$ 
of degree $\ge 2$.  Assume that $k$ is algebraically closed.
Then the exceptional collection $Y_1, Y_2, \dots$ 
is dense in $X$; in particular, it cannot be finite.
\end{example}

\begin{proof} We claim that $\trdeg_k k(X)^G < n$. Indeed, assume 
the contrary. Then the field extension $k(X)/k(X)^G$ is algebraic and
finitely generated; hence, it is finite. Now we can view
$\sigma^* \colon k(X) \to k(X)$ as a $k(X)^G$-linear transformation 
of a finite-dimensional $k(X)^G$-vector space $k(X)$. Since $\sigma$ is
injective, we conclude that it is also surjective, and thus 
$\sigma \colon X \to X$ has degree $1$, contradicting our choice of $\sigma$.

The claim tells us that the general fiber (and hence, every non-empty fiber)
of the map $\sigma \colon X \to Z$ has dimension $\geqslant 1$.   
Suppose $x \in X(k)$ is a periodic point of $X$, i.e., $\sigma^n(x) = x$
for some $n \geqslant 1$. Then
$G \cdot x$ is finite, and thus $\overline{G \cdot x}$ is $0$-dimensional.
Consequently, $G \cdot x$ cannot be dense in the fiber $\phi^{-1} (\phi(x))$, 
so $x$ has to lie in the exceptional locus $Y_1 \cup Y_2 \cup \dots$.
On the other hand,  
by a result of N.~Fakhruddin~\cite[Corollary 5.3]{fakhruddin} 
periodic $k$-points for $\sigma$ are dense in $X = \bbP^n$.
\end{proof}

\section{Rosenlicht's theorem revisited}

To put our proof of Theorem~\ref{thm.main} in perspective,
we will now show how our arguments can be modified 
to obtain a proof of Theorem~\ref{thm.rosenlicht-algebraic}
in the case, where $X \subset \bbP^n$ is an irreducible
quasi-projective $k$-variety.

An action of an algebraic group $G$ on $X$ given by a $k$-morphism 
\[ \begin{array}{c}
\psi \colon G \times X \to X \times X \subset X \times \bbP^n \\
(g, x) \mapsto (x, g(x)). 
\end{array} \]
Once again, we define $V \subset X \times \bbP^n$ to 
be the closure of the image 
of $\psi$, $\pi \colon V \to X$ to be projection 
to the first factor, and $V_x := \pi^{-1}(x)$ to be the fiber 
of $x \in X$, as in~\eqref{e.graph}.
The role of Proposition~\ref{prop.sections}
will now be played by Lemma~\ref{lem.rosenlicht} below. 
Note that the countable collection of exceptional subvarieties
$Y_i \subset X$ from Proposition~\ref{prop.sections}
does not come up, and the proof is considerably simpler 
in this setting.

\begin{lemma} \label{lem.rosenlicht}
There exists a $G$-invariant dense open subvariety
$U \subset X$ with the following property:
for any field $K/k$ and any $x \in U(K)$, $G \cdot x$ is dense in $V_x$.
\end{lemma}

\begin{proof} First we will prove the lemma under the assumption that $G$ 
is irreducible.  In this case $V = \overline{\psi(G \times X)}$ is also irreducible. 
By Chevalley's theorem, $\psi(G \times X)$ contains 
a dense open subvariety $W$ of $V$. After replacing $W$ by the union of 
its $g$-translates, as $g$ ranges over $G$, we may assume that 
$W$ is $G$-invariant.  Here $G$ acts on $V$ via the second factor. 
In particular, $W \cap V_x = \{ x \} \times (G \cdot x)$.
Set $C := V {\bf -} W$ and 
\[ \text{$D := \{ x \in X \, | \, C$ contains an irreducible component of $V_x$
\}.} \]
One readily checks that $D$ is a $G$-invariant closed subvariety of $X$.
We claim that $D \neq X$. Indeed, by the fiber dimension theorem
every irreducible component of $V_x$ has dimension $\geq \dim(V) - \dim(X)$.
If $D = X$, then $\dim(C) = \dim(V)$ and thus $C = V$, since $V$ is 
irreducible.
Consequently, $W = \emptyset$, a contradiction. This proves the claim. 
We conclude that the dense open subset 
$U := X {\bf -} D$ of $X$ has the desired property.
This completes the proof of the lemma under the assumption that $G$ 
is irreducible.  

In general denote the irreducible components of $G$ by 
$G^0, G^1, \dots, G^r$, where $G^0$ is
the identity component.  Let $V^0$ be the closure of the image of the action 
map $G^0 \times X \to X \times X$ in $X \times \bbP^n$. Then
\[ V = V^0 \cup (G^1 \cdot V^0) \cup \dots \cup (G^r \cdot V^0). \,  \]
As we have just shown, there exists a $G^0$-invariant dense open subset 
$U \subset X$ such that $G^0 \cdot x$ is dense in $V^0_{x}$ for any $x \in U$. 
Then clearly $G \cdot x$ is dense in $V_x$. 
A priori $U$ is only $G^0$-invariant, but we can make it $G$-invariant by
replacing it with the dense open subset 
$U^0 \cap (G^1 \cdot U^0) \cap \dots \cap (G^r \cdot U^0)$.
\end{proof}

>From now on we will replace $X$ by $U$, and thus assume that
$G \cdot x$ is dense in $V_x$ for every $x \in X$.
Once again, we 
define $X_0$ as the flat locus of the projection $\pi \colon V \to X$.
The same argument as in the proof of Theorem~\ref{thm.main}(d) shows that
$X_0$ is $G$-invariant. Once again, 
we define $\phi \colon X_0 \to \Hilb(n)$ as the morphism 
associated to $\pi^{-1}(X_0) \subset X \times \bbP^n$, which 
we view as a family of subvarieties of $\bbP^n$ parametrized by $X_0$.  
By construction $\phi$ separates the fibers of $\pi$; these fibers are 
$G$-orbit closures by Lemma~\ref{lem.rosenlicht}.
To complete the proof of Theorem~\ref{thm.rosenlicht-algebraic}
it remains to show that $\phi$ actually separates the orbits
in $X_0$, and not just their closures.  Note that since separating
orbit closures is the most we can do in the dynamical setting,
Lemma~\ref{lem.rosenlicht2} has no counterpart in our proof of
Theorem~\ref{thm.main}.

\begin{lemma} \label{lem.rosenlicht2}
Let $K/k$ be a field extension and $x \in X_0(K)$ be a $K$-point.
Then the orbit $G \cdot x$ is closed in $(X_0)_K$.
\end{lemma}

\begin{proof} After base-changing to the algebraic closure $\overline{K}$, we
may assume that $K = k$ is algebraically closed. Since
the morphism $\pi \colon V \to X$ is flat over $X_0$,
$\dim(V_x)$ does not depend on the choice of $x$ in $X_0(k)$. 
Thus the dense open subset $G \cdot x$ of $V_x$ also
has the same dimension for every $x \in X_0(k)$. 

Assume the contrary: $G \cdot x$ is not closed in $X_0$ for some 
$x \in X_0(k)$.  Then the complement to $G \cdot x$ in 
$\overline{G \cdot x}$ has a $k$-point $y \in X_0(k)$.  
Since $G \cdot x$ is open and dense in $\overline{G \cdot x}$,
we have $\dim(G \cdot y) < \dim(G \cdot x)$, a contradiction.
\end{proof}

This completes the proof of Theorem~\ref{thm.rosenlicht-algebraic}(a).
Part (b) is proved in exactly the same way as Theorem~\ref{thm.main}(c)
in Section~\ref{sect.main(c)}.

\bigskip
\noindent
{\sc Acknowledgements.}
The authors are grateful to E. Amerik and T.~Tucker for helpful correspondence.

\bibliographystyle{amsalpha}
\bibliography{Drinbib}

\def\cprime{$'$} \def\cprime{$'$}
\providecommand{\bysame}{\leavevmode\hbox to3em{\hrulefill}\thinspace}
\providecommand{\MR}{\relax\ifhmode\unskip\space\fi MR }
\providecommand{\MRhref}[2]{%
  \href{http://www.ams.org/mathscinet-getitem?mr=#1}{#2}
}
\providecommand{\href}[2]{#2}
\begin{thebibliography}{Kem07}

\bibitem[AC08]{Amerik-Campana}
E.~Amerik and F.~Campana, \emph{Fibrations m\'eromorphes sur certain
  vari\'et\'es \`a fibr\'e canonique trivial}, Pure Appl. Math. Quart.
  \textbf{4} (2008), no.~2, 1--37.

\bibitem[BGT]{recent-surface}
J.~P. Bell, D.~Ghioca, and T.~J. Tucker, \emph{Applications of $p$-adic
  analysis for bounding periods of subvarieties under etale maps}, Internat.
  Math. Res. Not. IMRN (2014), to appear.

\bibitem[BRS10]{BRS}
J.~P. Bell, D.~Rogalski, and S.~J. Sierra, \emph{The {D}ixmier-{M}oeglin
  equivalence for twisted homogeneous coordinate rings}, Israel J. Math.
  \textbf{180} (2010), 461--507.

\bibitem[Eis95]{eisenbud}
D.~Eisenbud, \emph{Commutative algebra}, Graduate Texts in Mathematics, vol.
  150, Springer-Verlag, New York, 1995, With a view toward algebraic geometry.

\bibitem[Fak03]{fakhruddin}
N.~Fakhruddin, \emph{Questions on self maps of algebraic varieties}, J.
  Ramanujan Math. Soc. \textbf{18} (2003), no.~2, 109--122. \MR{1995861
  (2004f:14038)}

\bibitem[Gro95]{grothendieck}
A.~Grothendieck, \emph{Techniques de construction et th\'eor\`emes d'existence
  en g\'eom\'etrie alg\'ebrique. {IV}. {L}es sch\'emas de {H}ilbert},
  S\'eminaire {B}ourbaki, {V}ol.\ 6, Soc. Math. France, Paris, 1995, pp.~Exp.\
  No.\ 221, 249--276. \MR{1611822}

\bibitem[Kem07]{kemper}
G.~Kemper, \emph{The computation of invariant fields and a constructive version
  of a theorem by {R}osenlicht}, Transform. Groups \textbf{12} (2007), no.~4,
  657--670. \MR{2365439 (2008m:13011)}

\bibitem[MS14]{medvedev-scanlon}
A.~Medvedev and T.~Scanlon, \emph{Invariant varieties for polynomial dynamical
  systems}, Ann. of Math. (2) \textbf{179} (2014), no.~1, 81--177. \MR{3126567}

\bibitem[PV94]{popov-vinberg}
V.~L. Popov and E.~B. Vinberg, \emph{Invariant theory}, Algebraic geometry,
  {IV}, Encyclopaedia of Mathematical Sciences, vol.~55, Springer-Verlag,
  Berlin, 1994, A translation of {{\i}t Algebraic geometry. 4} (Russian), Akad.
  Nauk SSSR Vsesoyuz. Inst. Nauchn. i Tekhn. Inform., Moscow, 1989 [ MR1100483
  (91k:14001)], Translation edited by A. N. Parshin and I. R. Shafarevich,
  pp.~vi+284. \MR{1309681 (95g:14002)}

\bibitem[Ros56]{rosenlicht}
M.~Rosenlicht, \emph{Some basic theorems on algebraic groups}, Amer. J. Math.
  \textbf{78} (1956), 401--443. \MR{0082183 (18,514a)}

\bibitem[Sei79]{seidenberg}
A.~Seidenberg, \emph{On the functions invariant under the action of an
  algebraic group on an algebraic variety}, Rend. Sem. Mat. Fis. Milano
  \textbf{49} (1979), 69--77 (1981). \MR{610602 (82e:14059)}

\bibitem[Zha06]{zhang-distributions}
S.-W. Zhang, \emph{Distributions in algebraic dynamics}, Surveys in
  differential geometry. {V}ol. {X}, Surv. Differ. Geom., vol.~10, Int. Press,
  Somerville, MA, 2006, pp.~381--430. \MR{2408228 (2009k:32016)}

\end{thebibliography}

\end{document}